\documentclass[a4paper,10pt]{article}
\usepackage{amsmath,amsthm,amssymb}
\usepackage{mathptmx}
\usepackage{graphicx,graphics}
\usepackage{tikz}
\newtheorem{theorem}{Theorem}[section]
\newtheorem{corollary}{Corollary}[section]
\newtheorem{lemma}{Lemma}[section]
\newtheorem{prop}{Proposition}[section]
\theoremstyle{definition}
\newtheorem{definition}{Definition}[section]
\newtheorem{observation}{Observation}[section]
\newtheorem{example}{Example}[section]
\newtheorem{problem}{Problem}[section]

\newtheorem{remark}{Remark}[section]
\usepackage{mathrsfs}
\usepackage{hyperref}
\usepackage{enumerate}
\usepackage{float}
\usepackage{graphicx}
\usetikzlibrary{shapes.geometric,arrows}
\usetikzlibrary{mindmap}
\numberwithin{equation}{section}
\usepackage{soul}

\author{\hspace{1cm} Eshwar Srinivasan \ and Ramesh Hariharasubramanian \\
{{\footnotesize s.eshwar@iitg.ac.in},\ {\footnotesize  ramesh\_h@iitg.ac.in}}\\{\footnotesize Department of Mathematics, Indian Institute of Technology Guwahati, Guwahati, Assam 781039, India}}

\begin{document}
\title{Minimum length word-representants of graph products}
\maketitle

\begin{abstract}
		A graph $G = (V, E)$ is said to be \textit{word-representable} if a word $w$ can be formed using the letters of the alphabet $V$ such that for every pair of vertices $x$ and $y$, $xy \in E$ if and only if $x$ and $y$ alternate in $w$. Gaetz and Ji have recently introduced the notion of minimum length word-representants for word-representable graphs. They have also determined the minimum possible length of the word-representants for certain classes of graphs, such as trees and cycles. It is know that Cartesian and Rooted products preserve word-representability. Moreover, Broere constructed a uniform word representing the Cartesian product of $G$ and $K_n$ using occurrence based functions.
	
	In this paper, we study the minimum length of word-representants for Cartesian and Rooted products using morphism and occurrence based function, respectively. Also, we solve an open problem posed by Broere in his master thesis. This problem asks to construct a word for the Cartesian product of two arbitrary word-representable graphs.
\end{abstract}
\textbf{Keywords: } word-representability, minimum length word-representant, cartesian product, rooted product.

\pagestyle{myheadings}
\markboth{Eshwar Srinivasan et.al. }{Minimum length word-representants of graph products}

\section{Introduction}

	The theory of word-representable graphs is a very promising research area and provides an interesting way to analyze and understand graphs using words. The notion of word-representable graphs was introduced by Kitaev, which was motivated by the study of Perkins semigroups in \cite{kitaev2008word}. After the introduction of this notion, many results have been done in this area. In \cite{gaetz2020enumeration}, Gaetz and Ji have recently introduced the notion of minimum length word-representants for word-representable graphs. The concept of minimum length of the word-representants of word-representable graphs are very interesting, as it relates structural properties of the graph with its minimum length word-representants. In \cite{SRINIVASAN2024149} (\textit{Theorem }\ref{di}), the authors have found a relation between the occurrence of letters in minimum length word-representant and the diameter of the graph. In this paper, we study the minimum length of word-representants for Cartesian and Rooted products using morphism and occurrence based function, respectively. Also, we solve an open problem posed by Broere in \cite{broere2018word}. This problem asks to construct a word for the Cartesian product of two arbitrary word-representable graphs.
	We begin with a brief review of basic definitions and results on word-representable graphs. We refer the reader to \cite{kitaev2017comprehensive} and \cite{kitaev2015words}, for a detailed treatment. All graphs considered here are simple and undirected.

Suppose that $w$ is a word over some alphabet, and $x$ and $y$ are two distinct letters in $w$. We say that $x$ and $y$ \textit{alternate} in $w$ if after deleting all letters except the copies of $x$ and $y$ in $w$, we either obtain a word $xyxy\ldots$ (of odd or even length) or a word $yxyx\ldots$ (of odd or even length). Hence by definition, if $w$ has a single occurrence of $x$ and a single occurrence of $y$, then $x$ and $y$ alternate in $w$.
\begin{definition}[\cite{kitaev2015words}, Definition 3.0.5]\label{wr}
	A graph $G = (V, E)$ is said to be \textit{word-representable} if a word $w$ can be formed using the letters of the alphabet $V$ such that for every pair of vertices $x$ and $y$, $xy \in E$ if and only if $x$ and $y$ alternate in $w$. We say that $w$ \textit{represents} $G$, and $w$ is called a \textit{word-representant} of $G$. Also, it is essential that $w$ contains each letter of $V$ at least once.
\end{definition}

A word is called \textit{k-uniform} if each letter occurs exactly $k$ times in it. A graph $G$ is \textit{$k$-word-representable} if it can be represented by a $k$-uniform word. The least $k$ for which a word-representant of a graph $G$ is $k$-uniform is called the \textit{representation number} of the graph $G$, and it is denoted by $\mathcal{R}(G)$.
\begin{theorem}[\cite{kitaev2008representable}, Theorem 7]
	A graph is word-representable if and only if it is $k$-word-representable for some $k$
\end{theorem}
\begin{prop}[\cite{kitaev2015words}, Proposition 3.2.7]\label{uv}
	Let $w = uv$ be a $k$-uniform word representing a graph $G$, where $u$ and $v$ are two, possibly empty, words. Then the word $w' = vu$ also represents $G$.
\end{prop}
\begin{definition}[\cite{kitaev2015words}, Definition 3.0.13]
	The reverse of the word $w = w_1w_2 \ldots w_n$ is the word \\$r(w) = w_n \ldots w_2 w_1$.
\end{definition}
\begin{prop}[\cite{kitaev2015words}, Proposition 3.0.14]\label{rw}
	If $w$ is a word-representant of a graph $G$, then $r(w)$ also represents the graph $G$.
\end{prop}

For a word $w$, suppose that $\pi(w)$ is the permutation obtained from $w$ after removing all but its leftmost occurrence of each letter $x$. We call $\pi(w)$ as the \textit{initial permutation} of $w$. Similarly, suppose that $\sigma(w)$ is the permutation obtained from $w$ after removing all but its rightmost occurrence of each letter $x$. We call $\sigma(w)$ as the \textit{final permutation} of $w$. Furthermore, a word $w$ restricted to certain letters $x_1, \ldots, x_m$ is denoted by $w|_{\{x_1, \ldots, x_m\}}$. For instance, if $w = 35423214$, then $\pi(w) = 35421$, $\sigma(w) = 53214$, and $w|_{\{1,2\}} = 221$.

\begin{observation}[\cite{kitaev2008representable}, Observation 4]\label{pw}
	Let $w$ be a word-representant of $G$. Then $\pi(w)w$ also represents $G$.
\end{observation}
\par In \cite{gaetz2020enumeration}, Gaetz and Ji have studied the absolute minimum length word-representants of certain classes of graphs, such as trees and cycles.
\begin{definition}[\cite{gaetz2020enumeration}, Definition 2.2]
	Let $G$ be a word-representable graph with $w$ being a word-representant of $G$, then $l(G)$ is defined as the minimum possible length of the word $w$.
\end{definition}
\begin{definition}
	Let $w$ be a word. $O(w, i)$ is defined as the set of letters of $w$ which occur exactly $i$ times in it. 
\end{definition}
\begin{definition}
	Let $w$ be a word-representant for the graph $G=(V, E)$. Let $x \in V$. Define $O_w(x)$ as the number of occurrences of the letter $x$ in the word $w$.
\end{definition}
\begin{example}
	In the word $w = 322414$, $O_w(3) = 1$ and $O_w(2) = 2$. $O(w,2) = \{2, 4\}$ and $O(w, 1) = \{1, 3\}$.
\end{example}
\begin{prop}[\cite{kitaev13}, Proposition 1]\label{xwx}
	Let $w = w_1xw_2xw_3$ be a word-representant of a graph $G$ such that $w_1$, $w_2$ and $w_3$ are possibly empty words, and $w_2$ contains no $x$. Then the possible neighbours of $x$ in $G$ are the letters in $O(w_2,1)$.
\end{prop}
\begin{lemma}[\cite{SRINIVASAN2024149}, Lemma 1.12]\label{kap}
	Let $w$ be a minimum length word-representant of a word-representable graph $G$. Then
	\begin{equation*}
		|O(w, 1)| \le \kappa_G.
	\end{equation*}
	where $\kappa_G$ is the size of a maximal clique of $G$.
\end{lemma}
\begin{definition}[\cite{SRINIVASAN2024149}, Definition 2.20]
	Let $w$ be a word. The minimum and maximum number of occurrences of a letter in $w$ are denoted by $O_{min}(w)$ and $O_{max}(w)$, respectively.
\end{definition}
	\begin{theorem}[\cite{SRINIVASAN2024149}, Theorem 2.22]\label{di}
	Let $w$ be a minimum length word-representant of a word-representable connected graph $G$. Then
	\begin{equation*}
		O_{max}(w) - O_{min}(w) \le diam(G)
	\end{equation*}
	where $diam(G)$ is the diameter of the graph $G$.
\end{theorem}
In this paper, we denote the concatenation of words $w_1w_2\ldots w_n$ as $\displaystyle\prod_{i =1}^{n} w_i$, where $w_i$ are possibly empty words. For any graph $G$, $|G|$ denotes the \textit{order} of the graph. For a word $w$, $|w|$ denotes the length of the word. In section 2, we find an upper bound for the minimum length of words representing the cartesian product of word-representable graphs in terms of their minimum length. Similarly, in section 3, we establish an upper bound for the minimum length of words representing the Rooted product of word-representable graphs in terms of their minimum length.

\section{Cartesian Products}
Cartesian product is a graph operation that preserves word-representability. In \cite{BroereZ19}, Broere and Zantema have constructed a uniform word representing the Cartesian product of word-representable graph with a complete graph using occurrence based function. In this section, we give an upper bound for the cartesian product of word-representable graphs $G$ with $K_2$, and $K_n$ in terms of $l(G)$. In addition, we solve an open problem \textit{(Question 6.10)} posed in \cite{broere2018word}, by finding a word representing the Cartesian product of two arbitrary word-representable graphs. Moreover, we give an upper bound for the minimum length of the word-representants of the Cartesian product of two arbitrary word-representable graphs. In this section, we will be using $u^v$ to denote the ordered pair of vertices, $(u,v) $, of $G \: \square \: H$, where $u \in V(G)$ and $v \in V(H)$.
\begin{definition}[\cite{kitaev2015words}, \textit{Definition 5.4.8}]
	The \textit{Cartesian product} of two graphs $G = (V(G), E(G))$ and $H = (V(H), E(H))$ is a graph $G \: \square \: H = (V(G \: \square \: H) , E(G \: \square \: H))$, where $V(G \: \square \: H) = V(G) \times V(H)$ and\\
	 $E(G \: \square \: H) = \{((u, u'), (v,v'))\: | \: u=u' \textit{ and } (u',v') \in E(H) \textit{ or } v=v' \textit{ and } (u,v) \in E(G)\}$.
\end{definition}
\begin{theorem}[\cite{kitaev2015words}, \textit{Theorem 5.4.10}]
    Let $G$ and $H$ be two word-representable graphs. Then the Cartesian product $G \: \square \: H$ is also word-representable.
\end{theorem}

The following lemmas will help us in proving the main results.
\begin{lemma}\label{xu}
    Let $w = xUxZ$ be a word-representant of a graph G, where $x \in V(G)$, such that $U$ is a word containing all elements of the set $V(G)\setminus\{x\}$. Then, $w' = UxZ$ also represents the graph $G$.  
\end{lemma}
\begin{proof}
    Let $y \in V(G)$. Suppose that $x$ is adjacent to $y$. This implies $y$ occurs exactly once in $U$, and $x$, $y$ alternate in $xZ$. Therefore, $w'|_{\{x,y\}} = yxZ|_{\{x,y\}}$. Hence, $x$ and $y$ alternate in $w'$ as well. 

    Suppose that $x$ is not adjacent to $y$. Then either $y$ occurs once or more than once in $U$, as $U$ contains all elements of $V(G) \setminus\{x\}$. If $y$ occurs only once in $U$, then $x$ and $y$ do not alternate $Z$ as they do not alternate in $w$. Hence, $x$ and $y$ do not alternate in $w' = UxZ$. If $y$ occurs more than once in $U$, then $x$ and $y$ do not alternate in $w'= UxZ$. So, $w' = UxZ$ also represents $G$.
\end{proof}
\begin{lemma}\label{vx}
    Let $w = UxZx$ be a word-representant of a graph G, where $x \in V(G)$, such that $Z$ is a word containing all elements of the set $V(G)\setminus\{x\}$. Then, $w' = UxZ$ also represents the graph $G$.  
\end{lemma}
\begin{proof}
    By \textit{Proposition }\ref{rw}, $r(w) = xr(Z)xr(U)$ also represents the graph $G$. Hence, by \textit{Lemma }\ref{xu}, $r(w') = r(Z)xr(U)$ also represents the graph $G$. Again by \textit{Proposition }\ref{rw}, $w' = UxZ$ also represents the graph $G$. 
\end{proof}

Before establishing the result for Cartesian products, let us introduce certain functions that will be used in the proof.
\begin{definition}
    Let $w_G = u_1u_2 \ldots u_n$ be a word-representant of graph $G$ and $w_H = v_1v_2 \ldots v_m$ be a word-representant of graph $H$.  Define a function $g: V(G)^* \times V(H)^* \rightarrow (V(G) \times V(H))^*$ as,
    \begin{equation*}
        g^{w_H}(w_G) = (u_1^{v_1}u_1^{v_2}\ldots u_1^{v_m}) \ldots (u_n^{v_1}u_n^{v_2}\ldots u_n^{v_m}).
    \end{equation*}
    where $u_i^{v_j} \in V(G) \times V(H)$.
\end{definition}
\begin{definition}
    Let $w_G = u_1u_2 \ldots u_n$ be a word-representant of graph $G$ and $w_H = v_1v_2 \ldots v_m$ be a word-representant of graph $H$.  Define a function $J: V(G)^* \times V(H)^* \rightarrow (V(G) \times V(H))^*$ as,
    \begin{equation*}
        J^{w_H}(w_G) = (u_1^{v_1}u_2^{v_1} \ldots u_n^{v_1}) \ldots (u_1^{v_m}u_2^{v_m} \ldots u_n^{v_m}).
    \end{equation*}
    where $u_i^{v_j} \in V(G) \times V(H)$.
\end{definition}
\begin{prop}\label{gj}
	Let $i$ be a letter and $w_1$ and $w_2$ be any two words. Then,
	$$g^i(w_1)J^i(w_2) = J^i(w_1w_2) = g^i(w_1w_2)$$
\end{prop}
\begin{proof}
	Let $w_1 = u_1u_2 \ldots u_n$ and $w_2 = v_1v_2 \ldots v_m$. Clearly, $$g^i(w_1)J^i(w_2) = (u_1^i) \ldots (u_n^i)(v_1^i \ldots v_m^i) = (u_1^i \ldots u_n^i)(v_1^i \ldots v_m^i) = J^i(w_1w_2) = g^i(w_1w_2)$$
\end{proof}

As mentioned in the begining of this section, the following results provide an upper bound for the minimum length of the word-representants of the Cartesian product of two word-representable graphs in terms of the minimum length of their word-representants. Hence, it suffices to find a word of certain length that represents the Cartesian product to establish an upper bound for the minimum length. But finding the word-representant is not trivial and easy. Thus, the main proof idea is to find a word using the morphism functions $g(w)$ and $J(w)$ as defined above. The proof ends by verifying whether the above formed word represents the Cartesian product or not.
\begin{theorem}\label{GK2}
    Let $G$ be a word-representable graph with minimum length of its word-representant $l(G)$. Then, minimum length of the word-representants of the graph $G \: \square \: K_2$,
    \begin{equation*}
        l(G \: \square \: K_2) \le 2l(G) + 3|G| - 2.
    \end{equation*}
\end{theorem}
\begin{proof}
    Let $w_G$ be a word-representant of the graph $G$. Let $V(G) = 
    \{u_1, u_2, \ldots, u_m\}$, $V(K_2) = \{1, 2\}$ and $u_j^{i} \in V(G) \times V(K_2)$. Let $\pi(w_G) = u_1u_2 \ldots u_m$ be the initial permutation of the word $w_G$. Let $w_{K_2} = 12$ be a word-representant of the graph $K_2$. We claim that the word $w_{G \: \square \: K_2} = g^{r(w_{K_2})}(\pi(w_G))J^2(\pi(w_G))g^{w_{K_2}}(w_G)$ represents the graph $G \: \square \: K_2$. By definition of the Cartesian product of graphs, two vertices $u_j^i$ and $u_k^l$ alternate in $w_{G \: \square \: K_2}$ if and only if either $j=k$ and $i$ and $l$ alternate in $w_{K_2}$ or $i=l$ and $u_j$ and $u_k$ alternate in $w_G$. 

    Suppose $j = k$. Then, $w_{G \: \square \: K_2}|_{\{u_j^1,u_j^2\}} = u_j^2u_j^1u_j^2g^{w_{K_2}}(w_G)|_{\{u_j^1,u_j^2\}}$. Hence, $u_j^1$ and $u_j^2$ alternate in $w_{G \: \square \: K_2}$ as $g^{w_{K_2}}(w_G)|_{\{u_j^1,u_j^2\}} = u_j^1u_j^2u_j^1u_j^2\ldots$. 
    
    Suppose $i=l$. Then, $w_{G \: \square \: K_2}|_{\{u_j^i,u_k^i\}} = (g^i(\pi(w_G))J^2(\pi(w_G))g^i(w_G))|_{\{u_j,u_k\}}$ = $g^1(\pi(w_G)w_G|_{\{u_j,u_k\}})$ if $i=1$ and $w_{G \: \square \: K_2}|_{\{u_j^i,u_k^i\}} = J^2(\pi(w_G)\pi(w_G)w_G|_{\{u_j,u_k\}})$ if $i=2$, by \textit{Proposition }\ref{gj}. Hence, by \textit{Observation }\ref{pw}, $u_j^i$ and $u_k^i$ alternate in $w_{G \: \square \: K_2}$ if and only if $u_j$ and $u_k$ alternate in $w_G$. 
    
    Suppose $j \not= k$ and $i \not= l$. Without loss of generality, suppose $j < k$ and $i = 1$, $l = 2$. Then $w_{G \: \square \: K_2}|_{\{u_j^1,u_k^2\}} = u_j^1u_k^2u_k^2g^{w_{K_2}}(w_G)|_{\{u_j^1,u_k^2\}}$. Hence, $u_j^1$ and $u_k^2$ do not alternate in $w_{G \: \square \: K_2}$. Moreover, if $i=2$ and $l=1$, we have $w_{G \: \square \: K_2}|_{\{u_j^2,u_k^1\}} = u_j^2u_k^1u_j^2g^{w_{K_2}}(w_G)|_{\{u_j^2,u_k^1\}}$. As $j<k$, $g^{w_{K_2}}(w_G)|_{\{u_j^2,u_k^1\}} = u_j^2u_k^1\ldots$. Therefore, $w_{G \: \square \: K_2}|_{\{u_j^2,u_k^1\}} = u_j^2u_k^1u_j^2u_j^2u_k^1\ldots$. Hence, $u_j^2$ and $u_k^1$ do not alternate in $w_{G \: \square \: K_2}$. As a result, $w_{G \: \square \: K_2}$ represents the graph $G \: \square \: K_2$. Therefore, $|w_{G \: \square \: K_2}|= 2l(G) + 3m$.
    
    Consider $w_{G \: \square \: K_2} = (u_1^2u_1^1u_2^2u_2^1 \ldots u_n^2u_n^1)(u_1^2 \ldots u_n^2)g^{w_{K_2}}(w_G)$. Since $g^{w_{K_2}}(w_G) = u_1^1u_1^2u_2^1u_2^2 \ldots$, by \textit{Lemma }\ref{xu}, $w'_{G \: \square \: K_2} = (u_2^2u_2^1u_3^2u_3^1 \ldots u_n^2u_n^1)(u_1^2 \ldots u_n^2)g^{w_{K_2}}(w_G)$ also represents $G \: \square \: K_2$. Hence,
    \begin{equation*}
        l(G \: \square \: K_2) \le |w'_{G \: \square \: K_2}| = 2l(G) + 3m - 2.
    \end{equation*}
\end{proof}
The following result gives us more tighter bound for the graph product $K_n \: \square \: K_2$.
\begin{theorem}\label{KnK2}
    The minimum length of the word-representants of the graph $K_n \: \square \: K_2$ for all $n \ge 2$,
    \begin{equation*}
        l(K_n \: \square \: K_2) \le 5n - 4.
    \end{equation*}
\end{theorem}
\begin{proof}
    We know that, from \textit{Theorem }\ref{GK2}, $w_{K_n \: \square \: K_2} = (2^b2^a3^b3^a \ldots n^bn^a)(1^b \ldots n^b)(1^a1^b \ldots n^an^b)$ represents the graph $K_n \: \square \: K_2$, where $w_{K_n} = 12 \ldots n$, $w_{K_2} = ab$ and $i^k \in V(K_n) \times V(K_2)$. By \textit{Lemma }\ref{vx}, $w'_{K_n \: \square \: K_2} = (2^b2^a3^b3^a \ldots n^bn^a)(1^b \ldots n^b)(1^a1^b \ldots (n-1)^a(n-1)^b)$ also represents the graph $K_n \: \square \: K_2$. Hence as $l(K_n) = n$, 
    \begin{equation*}
        l(K_n \: \square \: K_2) \le |w'_{K_n \: \square \: K_2}| = 5n - 4. 
    \end{equation*}
\end{proof}
The below example justifies the tightness of the upper bound established in the above theorem.
\begin{example}
    We know that $l(K_2 \: \square \: K_2) = l(C_4) = 6$. Based on \textit{Theorem }\ref{KnK2}, we have $l(K_2 \: \square \: K_2) \le 5(2) - 4 = 6$.
\end{example}
\begin{theorem}\label{GKn}
    Let $G$ be a word-representable graph with minimum length of its word-representant $l(G)$. Then, for all $n \ge 3$, the minimum length of the word-representants of the graph $G \: \square \: K_n$,
    \begin{equation*}
        l(G \: \square \: K_n) \le nl(G) + (n^2 -1)|G|.
    \end{equation*}
\end{theorem}
\begin{proof}
    Let $w_G$ be a word-representant of the graph $G$. Let $V(G) = 
    \{u_1, u_2, \ldots, u_m\}$, $V(K_n) = \{1, 2, \ldots, n\}$ and $u_j^{i} \in V(G) \times V(K_n)$. Let $\pi(w_G) = u_1u_2 \ldots u_m$ be the initial permutation of the word $w_G$. Let $\pi(w_{K_n},i) = i(i+1) \ldots n12 \ldots (i-1)$, where $w_{K_n} = 123 \ldots n$ is a word-representant of the graph $K_n$. Define a function $h^i(w_G) = g^{\pi(w_{K_n},i)}(w_G)$. Let,
    \begin{equation*}
        w_{G \: \square \: K_n} = h^2(\pi(w_G))J^2(\pi(w_G))h^3(\pi(w_G))\ldots J^{i-1}(\pi(w_G))h^i(\pi(w_G))\ldots J^n(\pi(w_G))h^1(w_G)
    \end{equation*}
    We claim that the word, $w_{G \: \square \: K_n}$, represents the graph $G \: \square \: K_n$. By the definition of the Cartesian product of graphs, two vertices $u_j^i$ and $u_k^l$ alternate in $w_{G \: \square \: K_n}$ if and only if either of the following cases holds. \begin{enumerate}
    \item[(i)] $j=k$ and $i$, $l$ alternate in $w_{K_n}$. \\
    \item[(ii)] $i=l$ and $u_j$, $u_k$ alternate in $w_G$.
    \end{enumerate}
   \textbf{Case (i): }
    Suppose $j = k$, and without loss of generality assume $1 < i < l < n$. Then, $$w_{G \: \square \: K_n}|_{\{u_j^i,u_j^l\}} = [(\displaystyle \prod_{r=2}^{r=i}h^r(\pi(w_G)))J^i(\pi(w_G))( \displaystyle \prod_{r=i+1}^{r=l}h^r(\pi(w_G)))J^l(\pi(w_G))(\displaystyle \prod_{r=l+1}^{r=n}h^r(\pi(w_G)))h^1(w_G)]|_{\{u_j^i,u_j^l\}}.$$
    Here,
    $$h^r(\pi(w_G))|_{\{u_j^i,u_j^l\}} = \begin{cases} u_j^iu_j^l& \text{if }r \le i,\\ u_j^lu_j^i& \text{if }i < r \le l,\\ u_j^iu_j^l& \text{if }r > l.\end{cases}$$ Moreover, $h^1(w_G)|_{\{u_j^i,u_j^l\}} = u_j^iu_j^l\ldots$.
    Hence, $w_{G \: \square \: K_n}|_{\{u_j^i, u_j^l\}} = (\displaystyle \prod_{r=2}^{r=i}u_j^iu_j^l)u_j^i(\displaystyle \prod_{r=i+1}^{r=l}u_j^lu_j^i)u_j^l(\displaystyle \prod_{r=l+1}^{r=n}u_j^iu_j^l)u_j^iu_j^l\ldots$.
    As a result, $u_j^i$ and $u_j^l$ alternate in $w_{G \: \square \: K_n}$ as $i$ and $l$ alternate in $w_{K_n}$.
    
    Suppose $j = k$ and $i = 1$ and $ l < n$. Then, $w_{G \: \square \: K_n}|_{\{u_j^1, u_j^l\}} = (\displaystyle \prod_{r=2}^{r=l}u_j^lu_j^1)u_j^l(\displaystyle \prod_{r=l+1}^{r=n}u_j^1u_j^l)u_j^1u_j^l\ldots.$
    Hence, $u_j^1$ and $u_j^l$ alternate in $w_{G \: \square \: K_n}$.
    
    Suppose $j = k$ and $i > 1$ and $ l = n$. Then, $w_{G \: \square \: K_n}|_{\{u_j^i, u_j^n\}} = (\displaystyle \prod_{r=2}^{r=i}u_j^iu_j^n)u_j^i(\displaystyle \prod_{r=i+1}^{r=n}u_j^nu_j^i)u_j^nu_j^iu_j^n\ldots.$
    Hence, $u_j^i$ and $u_j^n$ alternate in $w_{G \: \square \: K_n}$.
    
    Suppose $j = k$, $i = 1$ and $ l = n$. Then, $w_{G \: \square \: K_n}|_{\{u_j^1, u_j^n\}} = (\displaystyle \prod_{r=2}^{r=n}u_j^nu_j^1)u_j^nu_j^1u_j^n\ldots.$
    Hence, $u_j^1$ and $u_j^n$ alternate in $w_{G \: \square \: K_n}$.
    
    \noindent\textbf{Case (ii): }
    Suppose $i=l$, and  without loss of generality assume $j < k$. Then, $$w_{G \: \square \: K_n}|_{\{u_j^i,u_k^i\}} = [(\displaystyle \prod_{r=2}^{r=i}h^r(\pi(w_G)))J^i(\pi(w_G))(\displaystyle \prod_{r=i+1}^{r=n}h^r(\pi(w_G)))h^1(w_G)]|_{\{u_j^i,u_k^i\}}.$$
    For all $r$, $h^r(\pi(w_G))|_{\{u_j^i,u_k^i\}} = u_j^iu_k^i$. Hence, $w_{G \: \square \: K_n}|_{\{u_j^i, u_k^i\}} = (\displaystyle \prod_{r=2}^{r=i}u_j^iu_k^i)u_j^iu_k^i(\displaystyle \prod_{r=i+1}^{r=n}u_j^iu_k^i)g^i(w_G|_{\{u_j,u_k\}})$.
    Hence, $u_j^i$ and $u_k^i$ alternate in $w_{G \: \square \: K_n}$ if and only if they alternate in $w_G$.
    
    \noindent\textbf{Case (iii): } In this case, we check whether $u_j^i$, $u_k^l$ and $u_j^l$, $u_k^i$ do not alternate in $w_{G \: \square \: K_n}$ when $i \not= l$ and $j \not= k$.

    Suppose $i < l$, and  without loss of generality assume $j < k$. First, we consider $$w_{G \: \square \: K_n}|_{\{u_j^i,u_k^l\}} = [(\displaystyle \prod_{r=2}^{r=i}h^r(\pi(w_G)))J^i(\pi(w_G))(\displaystyle \prod_{r=i+1}^{r=l}h^r(\pi(w_G)))J^l(\pi(w_G))(\displaystyle \prod_{r=l+1}^{r=n}h^r(\pi(w_G)))h^1(w_G)]|_{\{u_j^i,u_k^l\}}.$$
    For all $r$, $h^r(\pi(w_G))|_{\{u_j^i,u_k^l\}} = u_j^iu_k^l.$ Hence, for all $1 < i < l < n$, $w_{G \: \square \: K_n}|_{\{u_j^i, u_k^l\}} = (\displaystyle \prod_{r=2}^{r=i}u_j^iu_k^l){u_j^i}(\displaystyle \prod_{r=i+1}^{r=l}u_j^iu_k^l)u_k^l\\(\displaystyle \prod_{r=l+1}^{r=n}u_j^iu_k^l)u_j^iu_k^l\ldots$
    Hence, $u_j^i$ and $u_k^l$ do not alternate in $w_{G \: \square \: K_n}$.
    
    Suppose $i = 1$ and $ l < n$. Then, $w_{G \: \square \: K_n}|_{\{u_j^1, u_k^l\}} = (\displaystyle \prod_{r=2}^{r=l}u_j^1u_k^l)u_k^l(\displaystyle \prod_{r=l+1}^{r=n}u_j^1u_k^l)u_j^1u_k^l\ldots.$
    Hence, $u_j^1$ and $u_k^l$ do not alternate in $w_{G \: \square \: K_n}$.
    
    Suppose $i > 1$ and $l = n$. Then, $w_{G \: \square \: K_n}|_{\{u_j^i, u_k^n\}} = (\displaystyle \prod_{r=2}^{r=i}u_j^iu_k^n)u_j^i(\displaystyle \prod_{r=i+1}^{r=n}u_j^iu_k^n)u_k^nu_j^iu_k^n\ldots.$
    Hence, $u_j^i$ and $u_k^n$ do not alternate in $w_{G \: \square \: K_n}$.
    
    Suppose $i = 1$ and $l = n$. Then, $w_{G \: \square \: K_n}|_{\{u_j^1, u_k^n\}} = (\displaystyle \prod_{r=2}^{r=n}u_j^1u_k^n)u_k^nu_j^1u_k^n\ldots.$
    Hence, $u_j^1$ and $u_k^n$ do not alternate in $w_{G \: \square \: K_n}$. 
    
    Now Consider, $w_{G \: \square \: K_n}|_{\{u_j^l,u_k^i\}} = [(\displaystyle \prod_{r=2}^{r=i}h^r(\pi(w_G)))J^i(\pi(w_G))(\displaystyle \prod_{r=i+1}^{r=l}h^r(\pi(w_G)))J^l(\pi(w_G))(\displaystyle \prod_{r=l+1}^{r=n}h^r(\pi(w_G)))\\h^1(w_G)]|_{\{u_j^l,u_k^i\}}.$
    For all $r$, $h^r(\pi(w_G))|_{\{u_j^i,u_k^i\}} = u_j^lu_k^i.$ Hence for all $1 < i < l < n$, $w_{G \: \square \: K_n}|_{\{u_j^l, u_k^i\}} = (\displaystyle \prod_{r=2}^{r=i}u_j^lu_k^i){u_k^i}\\(\displaystyle \prod_{r=i+1}^{r=l}u_j^lu_k^i)u_j^l(\displaystyle \prod_{r=l+1}^{r=n}u_j^lu_k^i)u_j^lu_k^i\ldots$
    Hence, $u_j^l$ and $u_k^i$ do not alternate in $w_{G \: \square \: K_n}$.
    
    Suppose $i = 1$ and $ l < n$. Then, $w_{G \: \square \: K_n}|_{\{u_j^l, u_k^1\}} = (\displaystyle \prod_{r=2}^{r=l}u_j^lu_k^1)u_j^l(\displaystyle \prod_{r=l+1}^{r=n}u_j^lu_k^1)u_j^lu_k^1\ldots.$
    Hence, $u_j^l$ and $u_k^1$ do not alternate in $w_{G \: \square \: K_n}$.
    
    Suppose $i > 1$ and $l = n$. Then, $w_{G \: \square \: K_n}|_{\{u_j^n, u_k^i\}} = (\displaystyle \prod_{r=2}^{r=i}u_j^nu_k^i)u_k^i(\displaystyle \prod_{r=i+1}^{r=n}u_j^nu_k^i)u_j^nu_j^nu_k^i\ldots.$
    Hence, $u_j^n$ and $u_k^i$ do not alternate in $w_{G \: \square \: K_n}$.
    
    Suppose $i = 1$ and $l = n$. Then, $w_{G \: \square \: K_n}|_{\{u_j^n, u_k^1\}} = (\displaystyle \prod_{r=2}^{r=n}u_j^nu_k^1)u_j^nu_j^nu_k^1\ldots.$
    Hence, $u_j^n$ and $u_k^1$ do not alternate in $w_{G \: \square \: K_n}$.

    Therefore, $w_{G \: \square \: K_n}$ represents the graph $G \: \square \: K_n$. Hence, $l(G \: \square \: K_n) \le |w_{G \: \square \: K_n}|$. As $|h^r(w_G)| = n|w_G|$ for all $r$,
    \begin{eqnarray*}
        l(G \: \square \: K_n) & \le & (n-1)|h^r(\pi(w_G))| + (n-1)l(\pi(w_G)) + |h^1(w_G)| \\ & = & (n-1)nm + (n-1)m + nl(G) \\ 
        & = & nl(G) + (n^2 - 1)m.
    \end{eqnarray*}
\end{proof}

\begin{theorem}\label{GH}
    Let $G$ and $H$ be two word-representable graphs such that $|G| \ge |H|$, with minimum length of their word-representants $l(G)$ and $l(H)$, respectively. Then minimum length of the word-representants of the graph $G \: \square \: H$,
    \begin{equation*}
        l(G \: \square \: H) \le |H|l(G) + |G|l(H) + (|H|^2 - 1)|G|.
    \end{equation*}
\end{theorem}
\begin{proof}
    Let $w_G$ be a word-representant of the graph $G$. Let $w_H$ be a word-representant of the graph $H$. Let $V(G) = 
    \{u_1, u_2, \ldots, u_m\}$, $V(H) = \{v_1,v_2, \ldots ,v_n\}$ and $u_j^{v_i} \in V(G) \times V(H)$. Let $\pi(w_G) = u_1u_2 \ldots u_m$ be the initial permutation and $\sigma(w_G) = t_1t_2\ldots t_m$ be the final permutation of the word $w_G$. Let $\pi(w_H) = v_1v_2 \ldots v_n$ be the initial permutation of the word $w_H$. Let $ \pi(w_H,i) = v_iv_{i+1} \ldots v_nv_1 \ldots v_{i-1}$. Let us define a function $h^i(w_G) = g^{\pi(w_H,i)}(w_G)$. Let,
    \begin{equation*}
        w_{G \: \square \: H} = h^2(\pi(w_G))\ldots J^{v_i}(\pi(w_G))h^{i+1}(\pi(w_G)) \ldots J^{v_n}(\pi(w_G))h^1(w_G)J^{w_H}(\sigma(w_G)).
    \end{equation*}
    We claim that the word, $w_{G \: \square \: H}$, represents the graph $G \: \square \: H$. By definition of the Cartesian product of graphs, two vertices $u_j^{v_i}$ and $u_k^{v_l}$ alternate in $w_{G \: \square \: H}$ if and only if either of the following cases holds. \begin{enumerate}
    \item[(i)] $j=k$ and $v_i$, $v_l$ alternate in $w_H$.\\
    \item[(ii)] $i=l$ and $u_j$, $u_k$ alternate in $w_G$.
    \end{enumerate} 
    \noindent\textbf{Case (i): }
    Suppose $j = k$ and consider $1<i<l<n$, without loss of generality. Then, $ w_{G \: \square \: H}|_{\{u_j^{v_i}, u_j^{v_l}\}}$ is,$$ [(\displaystyle \prod_{r=2}^{r=i}h^r(\pi(w_G)))J^{v_i}(\pi(w_G))(\displaystyle \prod_{r=i+1}^{r=l}h^r(\pi(w_G)))J^{v_l}(\pi(w_G))(\displaystyle \prod_{r=l+1}^{r=n}h^r(\pi(w_G)))h^1(w_G)J^{w_H}(\sigma(w_G))]|_{\{u_j^{v_i},u_j^{v_l}\}}. $$
    Here,
    $$h^r(\pi(w_G))|_{\{u_j^{v_i},u_j^{v_l}\}} = \begin{cases} u_j^{v_i}u_j^{v_l}& \text{if }r \le i,\\ u_j^{v_l}u_j^{v_i}& \text{if }i < r \le l,\\ u_j^{v_i}u_j^{v_l}& \text{if }r > l.\end{cases}$$ Further, $h^1(w_G)|_{\{u_j^{v_i},u_j^{v_l}\}} = u_j^{v_i}u_j^{v_l}\ldots$ and $J^{w_H}(\sigma(w_G))|_{\{u_j^{v_i}, u_j^{v_l}\}} = J^{w_H|_{\{v_i,v_l\}}}(u_j).$ Hence,
    $$w_{G \: \square \: H}|_{\{u_j^{v_i}, u_j^{v_l}\}} = (\displaystyle \prod_{r=2}^{r=i}u_j^{v_i}u_j^{v_l})u_j^{v_i}(\displaystyle \prod_{r=i+1}^{r=l}u_j^{v_l}u_j^{v_i})u_j^{v_l}(\displaystyle \prod_{r=l+1}^{r=n}u_j^{v_i}u_j^{v_l})(u_j^{v_i}u_j^{v_l}\ldots)J^{w_H|_{\{v_i,v_l\}}}(u_j).$$
    Hence, $u_j^{v_i}$ and $u_j^{v_l}$ alternate in $w_{G \: \square \: H}|_{\{u_j^{v_i}, u_j^{v_l}\}}$ if and only if $v_i$ and $v_l$ alternate in $w_H$.
    
    Suppose $j = k$ and $i = 1$ and $ l < n$. Then, $w_{G \: \square \: H}|_{\{u_j^{v_1}, u_j^{v_l}\}} = (\displaystyle \prod_{r=2}^{r=l}u_j^{v_l}u_j^{v_1})u_j^{v_l}(\displaystyle \prod_{r=l+1}^{r=n}u_j^{v_1}u_j^{v_l})(u_j^{v_1}u_j^{v_l}\ldots)\\J^{w_H|_{\{v_1,v_l\}}}(u_j).$
    Hence, $u_j^{v_1}$ and $u_j^{v_l}$ alternate in  $w_{G \: \square \: H}|_{\{u_j^{v_1}, u_j^{v_l}\}}$ if and only if $v_1$ and $v_l$ alternate in $w_H$.
    
    Suppose $j = k$ and $i > 1$ and $ l = n$. Then, $w_{G \: \square \: H}|_{\{u_j^{v_i}, u_j^{v_n}\}} = (\displaystyle \prod_{r=2}^{r=i}u_j^{v_i}u_j^{v_n})u_j^{v_i}(\displaystyle \prod_{r=i+1}^{r=n}u_j^{v_n}u_j^{v_i})u_j^{v_n}(u_j^{v_i}u_j^{v_n}\ldots)\\J^{w_H|_{\{v_i,v_n\}}}(u_j).$
    Hence, $u_j^{v_i}$ and $u_j^{v_n}$ alternate in $w_{G \: \square \: H}|_{\{u_j^{v_i}, u_j^{v_n}\}}$ if and only if $v_i$ and $v_n$ alternate in $w_H$.
    
    Suppose $j = k$ and $i = 1$ and $ l = n$. Then, $w_{G \: \square \: K_n}|_{\{u_j^{v_1}, u_j^{v_n}\}} = (\displaystyle \prod_{r=2}^{r=n}u_j^{v_n}u_j^{v_1})u_j^{v_n}(u_j^{v_1}u_j^{v_n}\ldots)\\J^{w_H|_{\{v_1,v_n\}}}(u_j).$
    Hence, $u_j^{v_1}$ and $u_j^{v_n}$ alternate in $w_{G \: \square \: H}|_{\{u_j^{v_1}, u_j^{v_n}\}}$ if and only if $v_1$ and $v_n$ alternate in $w_H$.
    
    \noindent\textbf{Case (ii): }
    Suppose $i=l$ and consider $j < k$, without loss of generality. Then, $ w_{G \: \square \: H}|_{\{u_j^{v_i}, u_k^{v_i}\}}$ is,$$ [(\displaystyle \prod_{r=2}^{r=i}h^r(\pi(w_G)))J^{v_i}(\pi(w_G))(\displaystyle \prod_{r=i+1}^{r=n}h^r(\pi(w_G)))h^1(w_G)J^{w_H}(\sigma(w_G))]|_{\{u_j^{v_i},u_k^{v_i}\}}. $$

    For all $r$, $h^r(\pi(w_G))|_{\{u_j^i,u_k^i\}} = u_j^{v_i}u_k^{v_i}.$ Further, $h^1(w_G)|_{\{u_j^{v_i},u_k^{v_i}\}} = g^{v_i}(w_G|_{\{u_j,u_k\}})$ and $J^{w_H}(\sigma(w_G))|_{\{u_j^{v_i},u_k^{v_i}\}} \\= J^{v_i}(\sigma(w_G)|_{\{u_j,u_k\}}).$ Then, by \textit{Proposition }\ref{gj}, $g^{v_i}(w_G|_{\{u_j,u_k\}})J^{v_i}(\sigma(w_G)|_{\{u_j,u_k\}}) = J^{v_i}((w_G\sigma(w_G))|_{\{u_j,u_k\}}).$ Hence,
    $$w_{G \: \square \: H}|_{\{u_j^{v_i}, u_k^{v_i}\}} = (\displaystyle \prod_{r=2}^{r=i}u_j^{v_i}u_k^{v_i})u_j^{v_i}u_k^{v_i}(\displaystyle \prod_{r=i+1}^{r=n}u_j^{v_i}u_k^{v_i})J^{v_i}((w_G\sigma(w_G)|_{\{u_j,u_k\}}).$$
    Hence, $u_j^{v_i}$ and $u_k^{v_i}$ alternate in $w_{G \: \square \: H}|_{\{u_j^{v_i}, u_k^{v_i}\}}$ if and only if $u_j$ and $u_k$ alternate in $w_G$. 
    
    \noindent\textbf{Case (iii): }
    In this case we check whether $u_j^{v_i}$, $u_k^{v_l}$ and $u_j^{v_l}$, $u_k^{v_i}$ do not alternate in $w_{G \: \square \: H}$ when $i \not= l$ and $j \not= k$.

    Suppose $i < l$ and $j < k$ without loss of generality. Then firstly we consider $w_{G \: \square \: H}|_{\{u_j^{v_i},u_k^{v_l}\}}$, which is, $$ [(\displaystyle \prod_{r=2}^{r=i}h^r(\pi(w_G)))J^{v_i}(\pi(w_G))(\displaystyle \prod_{r=i+1}^{r=l}h^r(\pi(w_G)))J^{v_l}(\pi(w_G))(\displaystyle \prod_{r=l+1}^{r=n}h^r(\pi(w_G)))h^1(w_G)J^{w_H}(\sigma(w_G))]|_{\{u_j^{v_i},u_k^{v_l}\}}. $$
    For all $r$, $h^r(\pi(w_G))|_{\{u_j^{v_i},u_k^{v_l}\}} = u_j^{v_i}u_k^{v_l}.$ Hence for all $1 < i < l < n$,
    $$w_{G \: \square \: H}|_{\{u_j^{v_i}, u_k^{v_l}\}} = (\displaystyle \prod_{r=2}^{r=i}u_j^{v_i}u_k^{v_l}){u_j^{v_i}}(\displaystyle \prod_{r=i+1}^{r=l}u_j^{v_i}u_k^{v_l})u_k^{v_l}(\displaystyle \prod_{r=l+1}^{r=n}u_j^{v_i}u_k^{v_l})g^{v_iv_l}(w_G|_{\{u_j,u_k\}})J^{w_H|_{\{v_i,v_l\}}}(\sigma(w_G)|_{\{u_j,u_k\}}).$$
    Hence, $u_j^{v_i}$ and $u_k^{v_l}$ do not alternate in $w_{G \: \square \: H}$.
    
    Suppose $i = 1$ and $ l < n$. Then, $w_{G \: \square \: H}|_{\{u_j^{v_1}, u_k^{v_l}\}} = (\displaystyle \prod_{r=2}^{r=l}u_j^{v_1}u_k^{v_l})u_k^{v_l}(\displaystyle \prod_{r=l+1}^{r=n}u_j^{v_1}u_k^{v_l})g^{v_1v_l}(w_G|_{\{u_j,u_k\}})\\J^{w_H|_{\{v_1,v_l\}}}(\sigma(w_G)|_{\{u_j,u_k\}}).$
    Hence, $u_j^{v_1}$ and $u_k^{v_l}$ do not alternate in $w_{G \: \square \: H}$.
    
    Suppose $i > 1$ and $l = n$. Then, $w_{G \: \square \: H}|_{\{u_j^{v_i}, u_k^{v_n}\}} = (\displaystyle \prod_{r=2}^{r=i}u_j^{v_i}u_k^{v_n}){u_j^{v_i}}(\displaystyle \prod_{r=i+1}^{r=n}u_j^{v_i}u_k^{v_n})u_k^{v_n}g^{v_iv_n}(w_G|_{\{u_j,u_k\}})\\J^{w_H|_{\{v_i,v_n\}}}(\sigma(w_G)|_{\{u_j,u_k\}}).$
    Hence, $u_j^{v_i}$ and $u_k^{v_n}$ do not alternate in $w_{G \: \square \: H}$.
    
    Suppose $i = 1$ and $l = n$. Then, $w_{G \: \square \: H}|_{\{u_j^{v_1}, u_k^{v_n}\}} = (\displaystyle \prod_{r=2}^{r=n}u_j^{v_1}u_k^{v_n})u_k^{v_n}g^{v_1v_n}(w_G|_{\{u_j,u_k\}})J^{w_H|_{\{v_1,v_n\}}}(\sigma(w_G)|_{\{u_j,u_k\}}).$
    Hence, $u_j^{v_1}$ and $u_k^{v_n}$ do not alternate in $w_{G \: \square \: H}$.
    
    Now let us consider $w_{G \: \square \: H}|_{\{u_j^{v_l},u_k^{v_i}\}}$, which is, $$ [(\displaystyle \prod_{r=2}^{r=i}h^r(\pi(w_G)))J^{v_i}(\pi(w_G))(\displaystyle \prod_{r=i+1}^{r=l}h^r(\pi(w_G)))J^{v_l}(\pi(w_G))(\displaystyle \prod_{r=l+1}^{r=n}h^r(\pi(w_G)))h^1(w_G)J^{w_H}(\sigma(w_G))]|_{\{u_j^{v_l},u_k^{v_i}\}}. $$
    For all $r$, $h^r(\pi(w_G))|_{\{u_j^{v_l},u_k^{v_i}\}} = u_j^{v_l}u_k^{v_i}.$ Hence, for all $1 < i < l < n$,
    $$w_{G \: \square \: H}|_{\{u_j^{v_l}, u_k^{v_i}\}} = (\displaystyle \prod_{r=2}^{r=i}u_j^{v_l}u_k^{v_i}){u_k^{v_i}}(\displaystyle \prod_{r=i+1}^{r=l}u_j^{v_l}u_k^{v_i})u_j^{v_l}(\displaystyle \prod_{r=l+1}^{r=n}u_j^{v_l}u_k^{v_i})g^{v_iv_l}(w_G|_{\{u_j,u_k\}})J^{w_H|_{\{v_i,v_l\}}}(\sigma(w_G)|_{\{u_j,u_k\}}).$$
    Hence, $u_j^{v_l}$ and $u_k^{v_i}$ do not alternate in $w_{G \: \square \: H}$
    
    Suppose $i = 1$ and $ l < n$. Then, $w_{G \: \square \: H}|_{\{u_j^{v_l}, u_k^{v_1}\}} = (\displaystyle \prod_{r=2}^{r=l}u_j^{v_l}u_k^{v_1})u_j^{v_l}(\displaystyle \prod_{r=l+1}^{r=n}u_j^{v_l}u_k^{v_1})g^{v_1v_l}(w_G|_{\{u_j,u_k\}})\\J^{w_H|_{\{v_1,v_l\}}}(\sigma(w_G)|_{\{u_j,u_k\}}).$
    Hence, $u_j^{v_l}$ and $u_k^{v_1}$ do not alternate in $w_{G \: \square \: H}$.
    
    Suppose $i > 1$ and $l = n$. Then, $w_{G \: \square \: H}|_{\{u_j^{v_n}, u_k^{v_i}\}} = (\displaystyle \prod_{r=2}^{r=i}u_j^{v_l}u_k^{v_i}){u_k^{v_i}}(\displaystyle \prod_{r=i+1}^{r=n}u_j^{v_n}u_k^{v_i})u_j^{v_n}g^{v_iv_n}(w_G|_{\{u_j,u_k\}})\\J^{w_H|_{\{v_i,v_n\}}}(\sigma(w_G)|_{\{u_j,u_k\}}).$
    Hence, $u_j^{v_n}$ and $u_k^{v_i}$ do not alternate in $w_{G \: \square \: H}$.
    
    Suppose $i = 1$ and $l = n$. Then, $w_{G \: \square \: H}|_{\{u_j^{v_n}, u_k^{v_1}\}} = (\displaystyle \prod_{r=2}^{r=n}u_j^{v_n}u_k^{v_1})u_j^{v_n}g^{v_1v_n}(w_G|_{\{u_j,u_k\}})J^{w_H|_{\{v_1,v_n\}}}(\sigma(w_G)|_{\{u_j,u_k\}}).$
    Hence, $u_j^{v_n}$ and $u_k^{v_1}$ do not alternate in $w_{G \: \square \: H}$ since $u_j^{v_n}$ occurs to the left of $u_k^{v_l}$ in $g^{v_1v_n}(w_G|_{\{u_j,u_k\}})$.

    Therefore, $w_{G \: \square \: H}$ represents the graph $G \: \square \: H$. Hence, $l(G \: \square \: H) \le |w_{G \: \square \: H}|.$ As $|h^r(w_G)| = nl(G)$ for all $r$ and $|J^{w_H}(w_G)| = l(H)l(G)$,
    \begin{eqnarray*}
        l(G \: \square \: H) & \le & (n-1)|h^r(\pi(w_G))| + \sum_{i = 2}^n|J^{v_i}(\pi(w_G))| + |h^1(w_G)| + |J^{w_H}(\sigma(w_G))| \\ & = & (n-1)nm + (n-1)m + nl(G) + l(H)m \\ & = & nl(G) + ml(H) + (n^2 -1)m.
    \end{eqnarray*}
\end{proof}
\begin{remark}
	If we calculate an upper bound for $G \: \square \: K_n$ using \textit{Theorem} \ref{GH}, then $l(G \: \square \: K_n) \le nl(G) + mn + (n^2 - 1)m$. But using \textit{Theorem }\ref{GKn}, we get $l(G \: \square \: K_n) \le nl(G) + (n^2-1)m$, which is much more tighter bound than the former one. Hence, we stated three different theorems for $G \: \square \: K_2$, $G \: \square \: K_n$ and $G \: \square \: H$.
\end{remark}
\begin{corollary}
    Let $G$ and $H$ be two word-representable graphs with representation number $k_1$ and $k_2$ respectively. Then, $G \: \square \: H$ is $(k_1 + k_2 + min \{|G|, |H|\})$-representable.
\end{corollary}
\begin{proof}
    Let $w_G$ and $w_H$ represents the graphs $G$ and $H$, respectively. Without loss of generality, let us assume that $|G| \ge |H|$. Let $V(G) = 
    \{u_1, u_2, \ldots, u_m\}$, $V(H) = \{v_1,v_2, \ldots ,v_n\}$ and $u_j^{v_i} \in V(G) \times V(H)$. Let $\pi(w_G) = u_1u_2 \ldots u_m$ be the initial permutation and $\sigma(w_G) = t_1t_2\ldots t_m$ be the final permutation of the word $w_G$. Let $\pi(w_H) = v_1v_2 \ldots v_n$ be the initial permutation of the word $w_H$. Let $ \pi(w_H,i) = v_iv_{i+1} \ldots v_nv_1 \ldots v_{i-1}$. Let us define a function $h^i(w_G) = g^{\pi(w_H,i)}(w_G)$. Let,
    \begin{equation*}
        w_{G \: \square \: H} = h^2(\pi(w_G))\ldots J^{v_i}(\pi(w_G))h^{i+1}(\pi(w_G)) \ldots J^{v_n}(\pi(w_G))h^1(w_G)J^{w_H}(\sigma(w_G)).
    \end{equation*}
    Then by \textit{Theorem }\ref{GH}, $w_{G \: \square \: H}$ represents the graph $G \: \square \: H$. Here, $O_{h^r(\pi(w_G))}(u_j^{v_i}) = 1$ for all $r$ and $u_j^{v_i}$. Similarly, $O_{J^{v_i}(\pi(w_G))}(u_j^{v_i}) = 1$ for all $i$ and $u_j^{v_i}$. Since the representation number of the graph $G$ is $k_1$, $O_{h^1(w_G)}(u_j^{v_i}) \le k_1$ for any $u_j^{v_i}$ and since the representation number of the graph $H$ is $k_2$, $O_{J^{w_H}(\sigma(w_G))}(u_j^{v_i}) \le k_2$ for any $u_j^{v_i}$. Hence for any $u_j^{v_i}$,
    \begin{eqnarray*}
        O_{w_{G \: \square \: H}}(u_j^{v_i}) & \le & \sum_{r=2}^n O_{h^r(\pi(w_G))}(u_j^{v_i}) + O_{J^{v_i}(\pi(w_G))}(u_j^{v_i}) + O_{h^1(w_G)}(u_j^{v_i}) + O_{J^{w_H}(\sigma(w_G))}(u_j^{v_i}) \\ & \le & n-1 + 1 + k_1 +k_2 \\ & = & k_1 + k_2 + n.
    \end{eqnarray*}
\end{proof}
\section{Rooted Products}

Another graph product that preserves word-representability is Rooted products. In this section, we calculate an upper bound for the minimum length of the word-representants of the Rooted product of a word-representable graph $G$ with $K_2$, $K_n$, and any arbitrary word-representable graph $H$.
\begin{definition}[\cite{kitaev2015words}, \textit{Definition 5.4.11}]
	The \textit{Rooted product} of a graph $G$ and a rooted graph $H$ is the graph $G \circ H$, which is defined as follows: take $|V(G)|$ copies of $H$ and for every vertex $v_i$ of $G$, identify $v_i$ with the root vertex of the $i$th copy of $H$.
\end{definition}
\begin{theorem}[\cite{kitaev2015words}, \textit{Theorem 5.4.13}]
    Let $G$ and $H$ be two word-representable graphs. Then the rooted product $G \circ H$ is also word-representable.
\end{theorem}
\begin{definition}[\cite{BroereZ19}, \textit{Definition 3}]
	Let $V$ and $V'$ be (possibly different) alphabets, and let $N_k = \{1, \ldots, k\}$.
	The labelling function of a word over $V$ is defined as $H: V^* \rightarrow (V \times N_k)^*$, where the $i$th occurrence of each letter $x$ is mapped to the pair $(x, i)$, and $k$ satisfies the property that every symbol occurs at most $k$ times in $w$. The word $H(w)$ is called the labelled version of $w$.
	An occurrence-based function is defined as applying a string homomorphism $h: (V \times N_k)^* \rightarrow (V')^*$ to an already labelled version of a word. As a shorthand we will write $h(w)$ instead of $h(H(w))$.
	
\end{definition}
\begin{theorem}
    Let $G$ be a word-representable graph with minimum length of its word-representant $l(G)$. Then, minimum length of the word-representants of the graph $G \circ K_2$,
    \begin{equation*}
        l(G \circ K_2) \le 2l(G) + \kappa_G.
    \end{equation*}
    where $\kappa_G$ is the size of the maximum clique of the graph $G$.
\end{theorem}
\begin{proof}
    Let $w_G$ be a word-representant of the graph $G$. Let $V(K_2) = \{r,1\}$, where $r$ is the rooted vertex. Let $x^r, x^1 \in V(G) \times V(K_2)$, where $x \in V(G)$. We claim that the word $w_{G \circ K_2} = h(w_G)$ represents the graph $G \circ K_2$, where\\
    
    If $O_{w_G}(x) = 1$
    \begin{equation*}
        h(x,1) = x^1x^rx^1
    \end{equation*}
    
    If $O_{w_G}(x) \ge 2$,
    \begin{equation*}
        h(x,i) = \begin{cases}
        x^r&  i = 1,\\
        x^1x^rx^1&  i = 2,\\
        x^rx^1&  i \ge 3.
        \end{cases}
    \end{equation*}

    By definition of the Rooted product of the graphs, two vertices $x^i$ and $y^j$ alternate in $w_{G \circ K_2}$ if and only if either $i = j = r$ and $x$ and $y$ alternate in $w_G$ or $x=y$ and $i$ and $j$ alternate in $w_{K_2}$. 

    Suppose $i = j = r$. Then, $$w_{G \circ K_2}|_{\{x^r, y^r\}} = h(w_G)|_{\{x^r, y^r\}}$$
    Since for all $x \in V(G)$, $x^r$ occurs exactly once for every occurrence of $x$ in $w_G$, $x^r$ and $y^r$ alternate in $w_{G \circ K_2}$ if and only if $x$ and $y$ alternate in $w_G$.
    
    Suppose $i \not= j$ and $x = y$. Then, $$w_{G \circ K_2}|_{\{x^r, x^1\}} = h(w_G)|_{\{x^r, x^1\}} = \begin{cases}
        x^1x^rx^1&  O_{w_G}(x) = 1,\\
        x^rx^1x^rx^1\ldots&  O_{w_G}(x) \ge 2.
    \end{cases}$$
    Hence, $x^r$ and $x^1$ alternate in $w_{G \circ K_2}$ as $r$ and $1$ alternate in $w_{K_2}$.

    Suppose $i \not= j$ and $x \not= y$. Then without loss of generality consider, $$w_{G \circ K_2}|_{\{x^r, y^1\}} = h(w_G)|_{\{x^r, y^1\}}$$
    Since between two $y^1$s only $y^r$ occurs exactly once, by \textit{Proposition }\ref{xwx}, $x^r$ and $y^1$ do not alternate in $w_{G \circ K_2}$.
    
    Therefore, $w_{G \circ K_2}$ represents the graph $G \circ K_2$. Hence,
    \begin{eqnarray*}
         l(G \circ K_2) & \le & |w_{G \circ K_2}|\\
         & = & \sum_{i=2}^{\mathcal{R}(G)} 2i|O(w_G, i)| + 3|O(w_G,1)|\\
         & = & 2\sum_{i=1}^{\mathcal{R}(G)} i|O(w_G, i)| + |O(w_G,1)|\\
         & = & 2l(G) + |O(w_G,1)|.
    \end{eqnarray*}
    Hence from \textit{Lemma }\ref{kap}, we get
    $$ l(G \circ K_2) \le 2l(G) + \kappa_G$$ where $\kappa_G$ is the size of the maximum clique of $G$.
\end{proof}
\begin{theorem}
    Let $G$ be a word-representable graph with minimum length of its word-representant $l(G)$. Then, minimum length of the word-representant of the graph $G \circ K_n$,
    \begin{equation*}
        l(G \circ K_n) \le nl(G) + (n-1)\kappa_G.
    \end{equation*}
    where $\kappa_G$ is the size of the maximum clique of the graph $G$.
\end{theorem}
\begin{proof}
    Let $w_G$ be a word-representant of the graph $G$. Let $V(K_n) = \{r,1,2, \ldots, n-1\}$, where $r$ is the rooted vertex. Therefore, $w_{K_n} = r12 \ldots (n-1)$ represents the graph $K_n$. Let $x^i \in V(G) \times V(K_2)$, where $x \in V(G)$ and $i \in V(K_n)$. We claim that the word $w_{G \circ K_n} = h(w_G)$ represents the graph $G \circ K_n$, where\\
     
    If $O_{w_G}(x) = 1$
    \begin{equation*}
        h(x,1) = x^1x^2\ldots x^{n-1}x^rx^1x^2\ldots x^{n-1}
    \end{equation*}

    If $O_{w_G}(x) \ge 2$,
    \begin{equation*}
        h(x,i) = \begin{cases}
        x^r&  i = 1,\\
        x^1x^2\ldots x^{n-1}x^rx^1x^2\ldots x^{n-1}&  i = 2,\\
        x^rx^1\ldots x^{n-1}&  i \ge 3.
        \end{cases}
    \end{equation*}
   
    By definition of the Rooted product of the graphs, two vertices $x^i$ and $y^j$ alternate in $w_{G \circ K_n}$ if and only if either $i = j = r$ and $x$ and $y$ alternate in $w_G$ or $x=y$ and $i$ and $j$ alternate in $w_{K_n}$. 

    Suppose $i = j = r$. Then, $$w_{G \circ K_n}|_{\{x^r, y^r\}} = h(w_G)|_{\{x^r, y^r\}}$$
    Since for all $x \in V(G)$, $x^r$ occurs exactly once for every occurrence of $x$ in $w_G$, $x^r$ and $y^r$ alternate in $w_{G \circ K_n}$ if and only if $x$ and $y$ alternate in $w_G$.

    Suppose $i \not= j$ and $x = y$. Without loss of generality consider $i < j$. Then,

    If $i \not= j \not= r$,
    $$w_{G \circ K_n}|_{\{x^i, x^j\}} = h(w_G)|_{\{x^i, x^j\}} = \begin{cases}
        x^ix^jx^ix^j&  O_{w_G}(x) = 1,\\
        x^ix^jx^ix^j\ldots&  O_{w_G}(x) \ge 2.
    \end{cases}$$

    If $j = r$,
    $$w_{G \circ K_n}|_{\{x^i, x^r\}} = h(w_G)|_{\{x^i, x^r\}} = \begin{cases}
        x^ix^rx^i &   O_{w_G}(x) = 1,\\
        x^rx^ix^rx^ix^r\ldots&  O_{w_G}(x) \ge 2.
    \end{cases}$$
    Hence, $x^i$ and $x^r$ alternate in $w_{G \circ K_n}$ as $i$ and $r$ alternate in $w_{K_n}$.
    
    Suppose $i \not= j$ and $x \not= y$. Then, $$w_{G \circ K_n}|_{\{x^i, y^j\}} = h(w_G)|_{\{x^i, y^j\}}$$
    Since between two $y^j$s only $y^r$ occurs exactly once, by \textit{Proposition }\ref{xwx}, $x^i$ and $y^j$ do not alternate in $w_{G \circ K_n}$.
    
    Therefore, $w_{G \circ K_n}$ represents the graph $G \circ K_n$. Hence,
    \begin{eqnarray*}
         l(G \circ K_n) & \le & |w_{G \circ K_n}|\\
         & = & \sum_{i=2}^{\mathcal{R}(G)} in|O(w_G, i)| + (2n-1)|O(w_G,1)|\\
         & = & n\sum_{i=1}^{\mathcal{R}(G)} i|O(w_G, i)| + (n-1)|O(w_G,1)|\\
         & = & nl(G) + (n-1)|O(w_G,1)|.
    \end{eqnarray*}
    Hence from \textit{Lemma }\ref{kap}, we get
    $$ l(G \circ K_n) \le nl(G) + (n-1)\kappa_G.$$ where $\kappa_G$ is the size of the maximum clique of $G$.
\end{proof}
\begin{theorem}
	Let $G$ and $H$ be two word-representable graphs, with minimum length of their word-representants $l(G)$ and $l(H)$, respectively. Then minimum length of the word-representants of the graph $G \circ H$,
	\begin{equation*}
		l(G \circ H) \le |H|l(G) + |G|l(H) + (n-1)\kappa_G.
	\end{equation*}
where $\kappa_G$ is the size of the maximum clique of the graph $G$.
\end{theorem}
\begin{proof}
	Let $w_G$ and $w_H$ be word-representants of the graphs $G$ and $H$, respectively. Let $V(H) = \{v_1, v_2, \ldots,\\ v_r, \ldots, v_n\}$, where $v_r$ is the rooted vertex. Let $\pi(w_H) = v_1v_2\ldots v_r \ldots v_n$ be the initial permutation of the word $w_H$. Let $x^{v_i} \in V(G) \times V(H)$, where $x \in V(G)$. We claim that the word $w_{G \circ H} = h(w_G)J^{w_H}(\sigma(w_G))$, represents the graph $G \circ H$, where
	
	If $O_{w_G}(x)  = 1$,
	\begin{equation*}
		h(x,1) = x^{v_{r+1}}\ldots x^{v_n}x^{v_1}\ldots x^{v_r}\ldots x^{v_n}x^{v_1} \ldots x^{v_{r-1}}
	\end{equation*}

	If $O_{w_G}(x) \ge 2$,
	\begin{equation*}
		h(x,i) = \begin{cases}
			x^r&  i = 1,\\
			x^{v_{r+1}}\ldots x^{v_n}x^{v_1}\ldots x^{v_r}\ldots x^{v_n}x^{v_1} \ldots x^{v_{r-1}}&  i = 2,\\
			x^{v_r}\ldots x^{v_n}x^{v_1} \ldots x^{v_{r-1}}&  i \ge 3.
		\end{cases}
	\end{equation*}

By definition of the Rooted product of the graphs, two vertices $x^{v_i}$ and $y^{v_j}$ alternate in $w_{G \circ H}$ if and only if either $i = j = r$ and $x$ and $y$ alternate in $w_G$ or $x = y$ and $v_i$ and $v_j$ alternate in $w_H$.

Suppose $i = j= r$. Then,
$$w_{G \circ H}|_{\{x^{v_r}, y^{v_r}\}} = h(w_G)J^{w_H}(\sigma(w_G))|_{\{x^{v_r}, y^{v_r}\}} = h(w_G)|_{\{x^{v_r}, y^{v_r}\}}J^{v_r}(\sigma(w_G)|_{\{x, y\}})$$
Since $x^{v_r}$ occurs exactly once in $h(w_G)$ in place of every occurrence of $x$ in $w_G$ for all $x \in V(G)$, $h(w_G)|_{\{x^{v_r}, y^{v_r}\}} \\= J^{v_r}(w_G|_{\{x, y\}})$. Therefore by \textit{Proposition }\ref{gj}, $w_{G \circ H} = J^{v_r}(w_G\sigma(w_G)|_{\{x, y\}})$. Hence,  $x^{v_r}$ and $y^{v_r}$ alternate in $w_{G \circ H}$ if and only if $x$ and $y$ in $w_G$.

Suppose $x = y$ and $i \not = j$. Then,
$$ w_{G \circ H}|_{\{x^{v_i}, x^{v_j}\}} = h(w_G)J^{w_H}(\sigma(w_G))|_{\{x^{v_i}, x^{v_j}\}} = h(w_G)|_{\{x^{v_i}, x^{v_j}\}}J^{w_H|_{\{v_i, v_j\}}}(x)$$
If $i \not= j \not= r$,
$$h(w_G)|_{\{x^{v_i}, x^{v_j}\}} = \begin{cases}
	x^{v_i}x^{v_j}x^{v_i}x^{v_j}&  O_{w_G}(x) = 1,\\
	x^{v_i}x^{v_j}x^{v_i}x^{v_j}\ldots&  O_{w_G}(x) \ge 2.
\end{cases}$$
If $j = r$,
$$h(w_G)|_{\{x^{v_i}, x^{v_r}\}} = \begin{cases}
	x^{v_i}x^{v_r}x^{v_i}&  O_{w_G}(x) = 1,\\
	x^{v_r}x^{v_i}x^{v_r}x^{v_i}\ldots&  O_{w_G}(x) \ge 2.
\end{cases} $$
Therefore in both the cases, $w_{G \circ H}|_{\{x^{v_i},x^{v_j}\}} = (x^{v_i}x^{v_j}x^{v_i}x^{v_j}\ldots)J^{w_H|_{\{v_i, v_j\}}}(x)$. Hence clearly, $x^{v_i}$ and $x^{v_j}$ alternate in $w_{G \circ H}$ if and only if $v_i$ and $v_j$ alternate in $w_H$.

Suppose $i \not= j$ and $x \not= y$.
$$w_{G \circ H}|_{\{x^{v_i}, y^{v_j}\}} = h(w_G)|_{\{x^{v_i}, y^{v_j}\}}J^{w_H|_{\{v_i, v_j\}}}(\sigma(w_G)|_{\{x, y\}})$$
Since between any two $y^{v_j}$s only $y^{v_r}$ occurs exactly once, by \textit{Proposition }\ref{xwx}, $x^{v_i}$ and $y^{v_j}$ do not alternate in $w_{G \circ H}$.

Therefore, $w_{G \circ H}$ represents the graph $G \circ H$. Hence as $|J^{w_H}(w_G)| = l(H)l(G)$, 
\begin{eqnarray*}
	l(G \circ H) & \le & |w_{G \circ H}|\\
	& = & \sum_{i = 2}^{\mathcal{R}(G)} in|O(w_G, i)| + (2n-1)|O(w_G,1)| + |G|l(H)\\
	& = & n\sum_{i=1}^{\mathcal{R}(G)} i|O(w_G, i)| + (n-1)|O(w_G,1)|\\
	& = & nl(G) + (n-1)|O(w_G,1)|.
\end{eqnarray*}
Hence from \textit{Lemma }\ref{kap}, we get $$l(G \circ H) \le |H|l(G) + |G|l(H) + (|H|-1)\kappa_G.$$
where $\kappa_G$ is the size of the maximum clique of the graph $G$.
\end{proof}
\section{Conclusion and Future Work}
In Section 2, we found a tighter upper bound for the minimum length of the word-representants of the graph $K_n \: \square \: K_2$. It would be interesting to strengthen the bound for the Cartesian product of an arbitrary word-representable graph with $K_2$ and with another arbitrary word-representable graph. This leads to the following problem.
\begin{problem}
	Find tighter bounds for $l(G \: \square \: H)$ in terms of $l(G)$ and $l(H)$ for various $H$ as follows
	\begin{itemize}
		\item  $H \cong K_2$.
		\item  $H \cong K_n$.
		\item  $H$ is arbitrary.
	\end{itemize}
\end{problem}

In \cite{gaetz2020enumeration}, Gaetz and Ji found lower bounds for $l(G)$ for some classes of graphs like \textit{triangle-free graphs}. It would be interesting find lower bounds for minimum length of word-representability preserving graph operations.
\begin{problem}
	Find a lower bound for $l(G')$, where $G'$ is a graph obtained from word-representable graphs $G$ and $H$ using some word-representability preserving graph operation.
\end{problem}
\bibliographystyle{plain}
\bibliography{ref.bib}
\end{document}